\documentclass[11pt,draft,amsmath,amscd,amsbsy,amssymb]{amsart}
\usepackage{amsmath,amsthm,latexsym,amscd,amsbsy,amssymb}
 \relax
\chardef\bslash=`\\ % p.  424, TeXbook %\newcommand{\ntt}{\seriesm\shape n\tt}

\makeatletter \def\verbatim{\interlinepenalty\@M \@verbatim
\leftskip\@totalleftmargin\advance\leftskip2pc
\frenchspacing\@vobeyspaces \@xverbatim} \makeatother \hfuzz1pc

\makeatletter \def\dgt@k{\dg@DX=-3 \dg@DY=2 \dg@SIZE=3}
\makeatother

\makeatletter \def\dgt@kk{\dg@DX=3 \dg@DY=-1 \dg@SIZE=3}% \makeatother

\theoremstyle{plain} \newtheorem{thm}{Theorem}[section]
\newtheorem{cor}[thm]{Corollary}
\newtheorem{lemma}[thm]{Lemma}

\theoremstyle{definition} 
\newtheorem{defin}[thm]{Definition} 

\usepackage[all]{xy}
\begin{document}

\title[Remarks on asymptotic power dimension]
{Remarks on asymptotic power dimension}
\author[J. Kucab]{Jacek Kucab}
\address{Faculty of Mathematics and Natural Sciences,
University of Rzesz\'ow, Rejtana 16 A, 35-310 Rze\-sz\'ow, Poland}
\email{mzar@litech.lviv.ua}
\email{}
\author[M. Zarichnyi]{ Michael Zarichnyi}
\address{Department of Mechanics and Mathematics,
Lviv National University, Universytetska Str. 1, 79000 Lviv, Ukraine}
\address{Faculty of Mathematics and Natural Sciences,
University of Rzesz\'ow, Rejtana 16 A, 35-310 Rze\-sz\'ow, Poland}
\email{mzar@litech.lviv.ua}

\thanks{}

\date{\today}
%\dedicatory{}
%\commby{}

%%% ----------------------------------------------------------------------

\begin{abstract} Using a result of Dranishnikov and Smith we prove that, under some conditions, the asymptotic power dimension of a proper metric space coincides with the dimension of its subpower corona.
\end{abstract}

%%% ----------------------------------------------------------------------
\maketitle
%%% ----------------------------------------------------------------------
\section{Introduction}

The asymptotic power dimension is defined in \cite{KZ}. In this note we use one result of \cite{KZ} in order to derive the coincidence of the asymptotic power dimension of a proper metric space and the dimension of the subpower corona (this corona is introduced in \cite{KZ1}) of this space.

\section{Preliminaries}

A metric space $X$ is called proper if every closed ball in $X$ is compact.

Let $D>0$. For a metric space $(X,\rho)$ and a family $\mathcal{U}$ of subsets of $X$ we
say that  $\mathcal{U}$ is $D$-bounded if for any $U\in \mathcal{U}$ we have
$\mathrm{diam}(U)\leq D$;  $\mathcal{U}$ is $D$-disjoint if for any $U,V\in \mathcal{U}$, $U\neq
V$ we have $\mathrm{inf}\lbrace \rho(x,y)\vert x\in U, y\in V\rbrace>D$;

Recall that a metric space $(X,\rho)$ is called
cocompact
if there is a compact subset
$K\subset
X$
such that
$X
=
\cup \{f(K)\mid f\in\mathrm{Isom}(X)\}$, where $\mathrm{Isom}(X)$ is the set of all isometries of
$X$.

A metric space $(X,\rho)$ is said to be $M$-connected, for some $M>0$ if for every $x,y\in X$ there exist $x=x_0,x_1,\dots,x_k=y$ such that $\{x_0,x_1,\dots,x_k\}\subset X$ and $\rho(x_i,x_{i+1})\le M$ for every $i=0,1,\dots,k-1$.

\begin{defin}
For a metric space $(X,\rho)$ and a nonnegative integer $n$ we say that the
asymptotic Assouad-Nagata dimension of $X$ does not exceed $n$
($\mathrm{asdim}_{AN}X\leq n$) if there are constants $c>0$ and $r_0>0$ such
that for any $r>r_0$ there exist $r$-disjoint and $cr$-bounded families
$\mathcal{U}_0, \mathcal{U}_1,\dots,\mathcal{U}_n$ of subsets of $X$ such that
$\bigcup_{i=0}^{n}\mathcal{U}_i$ is a cover of $X$.
\end{defin}
In a similar manner, one can introduce the notion of the asymptotic power dimension of a
metric space as follows \cite{KZ}:
\begin{defin}
For a metric space $(X,\rho)$ and a nonnegative integer $n$ we say that the
asymptotic power dimension of $X$ does not exceed $n$ ($\mathrm{asdim}_P X\leq
n$) if there are constants $\alpha>0$ and $r_0>0$ such that for any $r>r_0$
there exist $r$-disjoint and $r^\alpha$-bounded families $\mathcal{U}_0,
\mathcal{U}_1,...,\mathcal{U}_n$ of subsets of $X$ such that
$\bigcup_{i=0}^{n}\mathcal{U}_i$ is a cover of $X$.
\end{defin}

As usual we will say that $\mathrm{asdim}_P X= n$, for $n\ge1$, when $\mathrm{asdim}_P X\leq
n$ and it is not true that $\mathrm{asdim}_P X\leq n-1$.

It is well-known that for any metric $\rho$ on a set $X$ the function $\rho^\prime
: X\times X\to \mathbb{R}$, $\rho^\prime(x,y)=\mathrm{ln}(1+\rho(x,y))$, is also a
metric on $X$. In the sequel, we will denote $(X,\rho)$ by $X$ and $(X,\rho^\prime)$ by
$X^\prime$. Note that for a proper and unbounded metric space $X$,
the space $X^\prime$ is also proper and unbounded.

The following result is proved in \cite{KZ}; we provide the proof for the reader's convenience.
\begin{thm}
For any metric space $(X,\rho)$ we have $\mathrm{asdim}_P X=\mathrm{asdim}_{AN}
X^\prime$
\end{thm}
\begin{proof}
Let $\mathrm{asdim}_P X\leq n$. Let $\alpha$ and $r_0$ be as in the definition.
Let $c=2\alpha$ and let $r_0^\prime=\mathrm{max}(r_0,\frac{\mathrm{ln}2}
{\alpha})$. Let $r>r_0^\prime$. Then $e^r-1\geq r>r_0^\prime\geq r_0$ and thus there
exist $(e^r-1)$-disjoint($\rho$) and $(e^r-1)^\alpha$-bounded($\rho$) families
$\mathcal{U}_0, \mathcal{U}_1,\dots,\mathcal{U}_n$ of subsets of $X$ such that
$\bigcup_{i=0}^{n}\mathcal{U}_i$ is a cover of $X$. Let $x\in U$, $y\in V$ for
any $U,V\in \mathcal{U}_i$, $U\neq V$, and any $i\in\lbrace 0,1, \dots, n\rbrace$.
Then $$\rho^\prime(x,y)=\mathrm{ln}(1+\rho(x,y))>\mathrm{ln}(1+e^r-1)=r.$$ Let now
$x,y\in U$ for any $U\in \mathcal{U}_i$ and any $i\in\lbrace 0,1, \dots,
n\rbrace$. Then we have $$\rho^\prime(x,y)=\mathrm{ln}
(1+\rho(x,y))\leq\mathrm{ln}(1+(e^r-1)^\alpha)\leq \mathrm{ln}(1+e^{\alpha
r})\leq\mathrm{ln}(2e^{\alpha r})< \mathrm{ln}e^{2\alpha r}=cr,$$ for $\alpha
r>\mathrm{ln}2$, i.e.  $e^{\alpha r}>2$. Therefore for any $i\in\lbrace
0,1,\dots, n\rbrace$, the family $\mathcal{U}_i$ is $r$-disjoint($\rho^\prime$) and
$cr$-bounded($\rho^\prime$). Hence $\mathrm{asdim}_{AN}X^\prime\leq n$ and therefore
$\mathrm{asdim}_{AN}X^\prime\leq\mathrm{asdim}_P X$.

Let now
$\mathrm{asdim}_{AN} X^\prime\leq n$. Let $c$ and $r_0^\prime$ be as in the
definition. Let $\alpha=c+2^c$ and $r_0=e^{r_0^\prime}$. Let $r>r_0$. Hence
$r+1>r_0>1$, i.e. $\mathrm{ln}(r+1)>r_0^\prime$. Then there exist
$\mathrm{ln}(r+1)$-disjoint($\rho^\prime$) and $c\mathrm{ln}(r+1)$-bounded($\rho^\prime$) families $\mathcal{U}_0, \mathcal{U}_1,\dots,\mathcal{U}_n$
of subsets of $X$ such that $\bigcup_{i=0}^{n}\mathcal{U}_i$ is a cover of $X$.
Let $x\in U$, $y\in V$ for any $U,V\in \mathcal{U}_i$, $U\neq V$, and any
$i\in\lbrace 0,1, \dots, n\rbrace$. Then $$\rho(x,y)=e^{\rho^\prime(x,y)}-
1>e^{\mathrm{ln}(r+1)}-1=r.$$ Let now $x,y\in U$ for any $U\in \mathcal{U}_i$ and
any $i\in\lbrace 0,1,\dots, n\rbrace$. Then $$\rho(x,y)=e^{\rho^\prime(x,y)}-1\leq
e^{c\mathrm{ln}(r+1)}-1<(r+1)^c<(2r)^c<r^\alpha,$$ since for any $r$ we have
$\alpha>c+\mathrm{log}_r 2^c$ and hence $r^\alpha>r^c2^c$. Therefore for any
$i\in\lbrace 0,1, \dots, n\rbrace$ the family $\mathcal{U}_i$ is $r$-disjoint($\rho$) and $r^\alpha$-bounded($\rho$). Therefore $\mathrm{asdim}_P
X\leq n$. We conclude that  $\mathrm{asdim}_P X\leq\mathrm{asdim}_{AN} X^\prime$ and
therefore $\mathrm{asdim}_P X=\mathrm{asdim}_{AN} X^\prime$.
\end{proof}

The notions of sublinear (resp. subpower) corona are introduced in \cite{DS} and \cite{KZ1} respectively.  The necessary definitions are given below.
\begin{defin}
We recall that a function $p:\mathbb{R}_+\to\mathbb{R}_+$ is called
asymtotically subpower (resp. asymtotically sublinear) if for any $\alpha>0$ [$c>0$]
there exists $r_0>0$ such that for any $x>r_0$ statement $p(x)<x^\alpha$
(resp. $p(x)<cx$) holds.
\end{defin}
\begin{defin}
Let $(X,\rho)$ be an unbounded proper metric space with a basepoint $x_0$. We
say that a continous, bounded function $f:X\to\mathbb{R}$ is Higson subpower
(resp. Higson sublinear) if for any asymtotically subpower function $p$ (resp. asymtotically
sublinear function $p$) we have $\mathrm{lim}_{\vert
x\vert\to\infty}\mathrm{diam}(f(B_{p(\vert x\vert)}(x)))=0$, where $\vert
x\vert=\rho(x,x_0)$.
\end{defin}
It is not hard to see that Higson subpower functions as well as Higson sublinear
functions on a proper, unbounded metric space form  subalgebras of algebra of
all continuous and bounded functions on $X$. We will denote them by $CB_P(X)$ and
$CB_L(X)$ respectively.
\begin{thm}\label{t:1}
For an unbounded, proper metric space $(X,\rho)$ we have $CB_P(X)=CB_L(X^\prime)
$.
\end{thm}
\begin{proof}
Let $(X,\rho)$ be an unbounded proper metric space with a basepoint $x_0$.

First we are going to prove that $CB_P(X)\subset CB_L(X^\prime)$. Convention: we use $B'_r(x)$ for the balls in $X'$ and also $|x|'=\rho(x,x_0)$.

Let $f\in CB_P(X)$ and let $\phi\colon\mathbb{R}_+\to\mathbb{R}_+$ be any
asymptotically sublinear function.
Note that \begin{align*} B^\prime_{\phi(\vert x\vert^\prime)}(x)=&\lbrace y\in X\vert
\rho^\prime(y,x)<\phi(\vert x\vert^\prime)\rbrace\\ =&\lbrace y\in
X\mid\ln(1+\rho(y,x))<\phi(\ln(1+\vert x\vert))\rbrace\\ =&
\lbrace y\in X\mid \rho(y,x)<e^{\phi(\ln(1+\vert x\vert))}-1\rbrace.\end{align*} Denoting
by $\psi(\vert x\vert)=e^{\phi(\ln(1+\vert x\vert))}-1$, the last set is just
$B_{\psi(\vert x\vert)}(x)$.
Now for every $n\in\mathbb{N}$ there is some $c_n\in\mathbb{R}$ that
$$\ln\psi(\vert x\vert)<\ln e^{\phi(\ln(1+\vert x\vert))}=\phi(\ln(1+\vert
x\vert))<\frac{1}{n}\ln(1+\vert x\vert)),$$ whenever $\ln(1+\vert x\vert))>c_n$
(or equivalently $\vert x\vert>e^{c_n}-1$). The latter means that for every
$n\in\mathbb{N}$ there is  $c_n\in\mathbb{R}$ such that $\psi(\vert
x\vert)<(1+\vert x\vert)^{1/n}<\vert x\vert^{2/n}$ whenever $\vert
x\vert>\max(2,e^{c_n}-1)$ and so $\psi$ is an asymptotically subpower function
(assuming by definition $\psi(\lambda)=1$ for those $\lambda\in \mathbb{R}_+$
for which there is no $x\in X$, such that $\lambda=\vert x\vert$). Thus
$\mathrm{diam}(f(B^\prime_{\phi(\vert x\vert^\prime)}(x)))=\mathrm{diam}
(f(B_{\psi(\vert x\vert)}(x)))\to 0$. We conclude that $f\in CB_L(X^\prime)$.

Now we are going to prove that $CB_L(X^\prime)\subset CB_P(X)$.

Let $f\in CB_L(X^\prime)$ and let $\phi\colon\mathbb{R}_+\to\mathbb{R}_+$ be an
asymptotically  subpower function. Note that $\phi+1$ is an asymptotically
subpower function as well. Let $c_1>0$ be large enough, such that $1+\phi(x)<x<x+1$ for
every $x>c_1$ and for $n\in\mathbb{N}$, $n>1$, let $c_n>c_{n-1}$ be large enough,
such that $1+\phi(x)<x^{1/n}<(1+x)^{1/n}$ for every $x>c_n$. We will now construct a
function $\psi\colon \mathbb{R}_+\to\mathbb{R}_+$ as follows: $\psi(x)=x/n$ for $x\in
(c_n,c_{n+1}]$, $n=1,2,\dots$, and $\psi(x)=1$ for $x\in(0,c_1]$. The fact that
$\psi$ is an asymptotically sublinear function is obvious, as for every
$n\in\mathbb{N}$ there exists $c_{n+1}$ such that for every $x>c_{n+1}$ we have
$\psi(x)\leq\frac{1}{n+1}x<\frac{1}{n}x$. We will now show that for $x\in X$
with $\vert x\vert$ large enough, $B_{\phi(\vert x\vert)}(x)\subset
B^\prime_{\psi(\vert x\vert^\prime)}(x)$. Let then $y\in
B_{\phi(\vert x\vert)}(x)$. Therefore $\rho(x,y)<\phi(\vert x\vert)$. Hence
$$\rho^\prime(x,y)=\ln(1+\rho(x,y))<\ln(1+\phi(\vert x\vert))<\ln(1+\vert
x\vert)^{1/n}=\frac{1}{n}\ln(1+\vert x\vert)$$ whenever $\vert x\vert>c_n$, hence
$\rho^\prime(x,y)<\psi(\vert x\vert^\prime)$ and $y\in B^\prime_{\psi(\vert
x\vert^\prime)}(x)$ for $\vert x\vert$ large enough.

Since $\mathrm{diam}
(f(B^\prime_{\psi(\vert x\vert^\prime)}(x))\to 0$, we see that $\mathrm{diam}
(f(B_{\phi(\vert x\vert)}(x))$ and hence $f\in CB_P(X)$.
\end{proof}
Let $h_P X$ and $h_L X$ denote the compactifications of $X$ with respect to the algebras $CB_P(X)
$ and $CB_L(X)$ respectively. Let $\nu_P X$ and $\nu_L X$ denote the coronas of
these compactifications.
\begin{cor}
For an unbounded proper metric space $(X,\rho)$ we have $\dim\nu_P X=\dim\nu_L
X^\prime$.
\end{cor}
\begin{proof}
Let $(X,\rho)$ be an unbounded proper metric space. By  Theorem \ref{t:1} we
have $CB_P(X)=CB_L(X^\prime)$ and hence there is a homeomorphism $h_P X\to h_L
X^\prime$ which extends the identity and this homeomorphism induces a
homeomorphism on the coronas.
\end{proof}
\begin{lemma}
For a metric space $(X,\rho)$ a set $\Gamma$ of isometries of $X$ coincides with the set $\Gamma^\prime$ of isometries of $X^\prime$.
\end{lemma}
\begin{proof}
For any $\gamma\in\Gamma$ and any $x,y\in X$ we have
$$\rho^\prime(\gamma(x),\gamma(y))=\mathrm{ln}
(1+\rho(\gamma(x),\gamma(y)))=\mathrm{ln}(1+\rho(x,y))=\rho^\prime(x,y),$$ so
$\Gamma\subset\Gamma^\prime$ and for any $\gamma^\prime\in\Gamma^\prime$ and any
$x,y\in X$ we have
$$\rho(\gamma^\prime(x),\gamma^\prime(y))=e^{\rho^\prime(\gamma^\prime(x),\gamma^
\prime(y))}-1=e^{\rho^\prime(x,y)}-1=\rho(x,y),$$ therefore $\Gamma^\prime\subset\Gamma$, i.e.
 $\Gamma=\Gamma^\prime$.
\end{proof}
\begin{cor}
If $X$ is a cocompact metric space, so is $X^\prime$.
\end{cor}

The following is proved in \cite{DS} (see Corollary 3.11 therein).
\begin{thm}
For a cocompact, proper, unbounded  and $M$-connected (for some $M>0$) metric space $X$ with
finite $\mathrm{asdim}_{AN} X$ we have
\begin{center}
$\mathrm{asdim}_{AN} X=\dim\nu_L X$
\end{center}
\end{thm}
\begin{cor}\label{c:1}
For a proper, unbounded, cocompact and $M$-connected  (for some $M>0$) metric space $X$ with
finite $\mathrm{asdim}_P X$ we have
\begin{center}
$\mathrm{asdim}_P X=\dim\nu_P X$.
\end{center}
\end{cor}
\begin{proof}
Let $(X,\rho)$ be a proper, unbounded, cocompact and $M$-connected (for some
$M>0$) metric space with finite $\mathrm{asdim}_P X$. Then $X^\prime$ is proper,
unbounded and cocompact with finite $\mathrm{asdim}_{AN} X^\prime$. Since for any
$x,y\in X$ we have $\rho^\prime(x,y)=\ln(1+\rho(x,y))\leq \rho(x,y)$, the space
$X^\prime$ is $M$-connected as well.
Using previous results we have:
\begin{center}
$\mathrm{asdim}_P X=\mathrm{asdim}_{AN} X^\prime=\dim\nu_L X^\prime=\dim\nu_P X$
\end{center}
\end{proof}
\begin{cor}
For a finitely generated group $G$ with word metric and of finite subpower
asymptotic dimension we have:
\begin{center}
$\mathrm{asdim}_P G=\dim\nu_P G$.
\end{center}
\end{cor}

\section{Remarks}

Corollary \ref{c:1} provides an answer to one question formulated in \cite{KZ}. 

The remetrization $\rho\mapsto\rho'$ can be used also for reducing some questions concerning the asymptotic power dimension to the corresponding questions concerning the asymptotic Assouad-Nagata dimension.   In particular, we conjecture that this remetrization would be useful in finding of an alternative proof of a characterization of the asymptotic power dimension in terms of mappings into polyhedra (this was obtained by the first author).

% ------------------------------------------------------------------------
%\bibliographystyle{amsplain}
%\bibliography{99}

\end{document}